\documentclass[final,3p,times,12pt]{elsarticle}
\usepackage[english]{babel}
\usepackage{amsthm}
\usepackage{amsmath,amssymb,amsfonts}
\usepackage[dvipsnames,usenames]{color}
\usepackage[utf8]{inputenc}
\usepackage{tikz}
\usepackage{algorithm}
\usetikzlibrary{positioning}
\usepackage{caption}
\usepackage{subcaption}

\newtheorem{thm}{Theorem}[section]
\newtheorem{lem}[thm]{Lemma}
\newtheorem{cor}[thm]{Corollary}
\newtheorem{claim}{Claim}

\def\Z{\mathbb Z}

\def\R{\mathbb R}
\def\diam{\textrm{diam}}

\journal{European J. Combin.} 

\begin{document}

\begin{frontmatter}

\title{Mirror graphs: graph theoretical characterization of reflection arrangements and finite Coxeter groups}

\author[LJ]{Tilen Marc}
				\ead{tilen.marc@imfm.si}
				
\address[LJ]{Institute of Mathematics, Physics, and Mechanics, Jadranska 19, 1000 Ljubljana, Slovenia}

\begin{abstract}
Mirror graphs were introduced by Brešar et al.~in 2004 as an intriguing class of graphs: vertex-transitive, isometrically embeddable into hypercubes, having a strong connection with regular maps and polytope structure. In this article we settle the structure of mirror graphs by characterizing them as precisely the Cayley graphs of the finite Coxeter groups or equivalently the tope graphs of reflection arrangements -- well understood and classified structures. We provide a polynomial algorithm for their recognition.
\end{abstract}
\end{frontmatter}
\sloppy
\section{Introduction and preliminaries}
%
%

Brešar, Klavžar, Lipovec, and Mohar gave the following definition in \cite{brevsar2004cubic}. Let $G$ be a simple, connected graph. Call a partition $P = \{E_1, E_2,\ldots, E_k\}$ of  edges in $G$ a \emph{mirror partition} if for every $i\in \{1,\ldots ,k\}$, there exists an automorphism $\alpha_i$ of $G$ such that:

\begin{enumerate}[(i)]\label{def_mirror_auto}
\item for every edge $uv \in E_i$: $\alpha_i(u) = v$ and $\alpha_i(v) = u$
\item $G - E_i$ consists of two connected components $G_i^1$ and $G_i^2$, and $\alpha_i$ maps $G_i^1$
isomorphically onto $G_i^2$
\end{enumerate}
A graph that has a mirror partition is called a \emph{mirror graph.} By definition they are highly symmetrical graphs. In \cite{brevsar2004cubic} it was shown that all mirror graphs are vertex-transitive, and certain connections with regular maps and polytope structures were established, indicating strong geometric properties of these graphs.

A more surprising result that they provided is that every mirror graph can be isometrically embedded (in the shortest path metric) into a hypercube graph $Q_d$, where $Q_d$ is a graph whose vertices are vectors in $\{1,0\}^d$ and two vertices are adjacent if they differ in exactly one coordinate. Graphs with this property are called \textit{partial cubes}. In particular a mirror partition of edges in a mirror graph must coincide with the coordinate partition of the edges in the embedding. This implies that classes $E_1, E_2,\ldots, E_k$ can be recognized by the so called relation $\Theta$ on the edges of graph $G$ defined as follows: $ab \Theta xy$ if $d(a,x) + d(b,y)\neq d(a,y) + d(b,x)$, where $d$ is the shortest path distance function. Thus a mirror partition, if it exists, is unique and easily computable based on the metric of the graph. It will follow from our results that also mirror automorphisms can be recognized efficiently.

Since partial cubes have inherent metric properties, the connection described above can be used to to better understand both classes of graphs. In fact, one of the motivations of mirror graphs was to build examples of cubic partial cubes -- intriguing class of graphs with many surprising properties. Nevertheless, the connection does not directly explain geometric properties exposed in examples of mirror graphs of \cite{brevsar2004cubic}. Even vertex-transitivity of mirror graphs is not a characterizing properties since vertex-transitive partial cubes are only classified in the cubic case \cite{marc2015classification} and not all of them are mirror graphs. In this paper we will expose a connection between mirror graphs and (\emph{realizable}) \emph{oriented matroids}, explaining the geometric properties of the former. For a general definition of oriented matroids and their connections to graphs see \cite{bandelt2015coms,MR1744046}; in this paper we will limit ourselves to realizable oriented matroids, for the sake of simplicity.  

One way to describe them is to consider a set of $m>0$ pairwise different hyperplanes $\{H_1,\ldots , H_m\}$ in $\mathbb{R}^n$, for some $n>0$, all incident with the origin of the space. Such a \emph{hyperplane arrangement} cuts $\mathbb{R}^n$ into connected spaces called \emph{chambers}. The \emph{tope graph} of a hyperplane arrangement is a graph whose vertices are chambers and two chambers are adjacent if they are separated by a single hyperplane. Tope graphs can be isometrically embedded in hypercubes \cite[Proposition 4.2.3]{MR1744046}, but the reverse problem of characterizing graphs which are tope graphs of hyperplane arrangements is an open problem. A more broad problem of characterizing graphs which correspond to arrangements of pseudohyperplanes in a real projective space is equivalent to characterizing the tope graphs of oriented matroids and was answered in \cite{da1995axioms}.

%
%

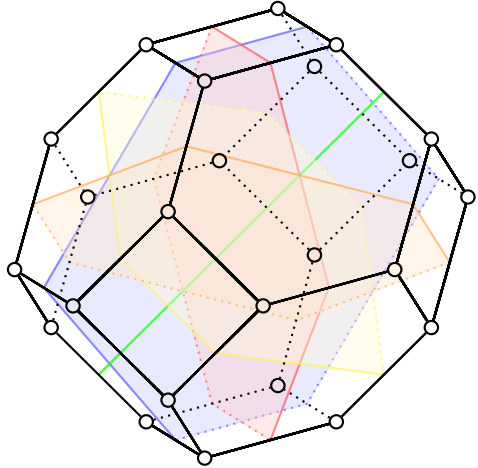
\begin{figure}[t]
\begin{center}
\begin{tikzpicture}[scale=1.25,thick,fill opacity=.4,draw opacity=1]
\tikzstyle{vertex}=[draw, circle, inner sep=0pt, minimum size =5pt]

\filldraw[
        draw=blue!20, draw opacity=0.4,%
        fill=blue!20,%
    ]          (-0.5,2,0.5)
            -- (0.5,2,-0.5)
            -- (1.5,0,-1.5)
            -- (0.5,-2,-0.5)
            -- (-0.5,-2,0.5)
            -- (-1.5,0,1.5)
            -- cycle;

\draw[
        draw=blue, draw opacity=0.4, dotted
    ] (0.5,2,-0.5)
            -- (1.5,0,-1.5)
            -- (0.5,-2,-0.5)
            -- (-0.5,-2,0.5);

\draw[
        draw=blue, draw opacity=0.4,
    ]
            (0.5,2,-0.5)
            -- (-0.5,2,0.5)
            -- (-1.5,0,1.5)
            -- (-0.5,-2,0.5);

\filldraw[
        draw=red!20, draw opacity=0.4,%
        fill=red!20,%
    ]          (0.5,2,0.5)
            -- (-0.5,2,-0.5)
            -- (-1.5,0,-1.5)
            -- (-0.5,-2,-0.5)
            -- (0.5,-2,0.5)
            -- (1.5,0,1.5)
            -- cycle;

\draw[
        draw=red, draw opacity=0.4, dotted
    ] (-0.5,2,-0.5)
            -- (-1.5,0,-1.5)
            -- (-0.5,-2,-0.5)
            -- (0.5,-2,0.5);

\draw[
        draw=red, draw opacity=0.4,
    ]
            (-0.5,2,-0.5)
            -- (0.5,2,0.5)
            -- (1.5,0,1.5)
            -- (0.5,-2,0.5);
            
\filldraw[
        draw=green, draw opacity=0.4,%
        fill=green!20,%
    ]          (1.5,1.5,0)
            -- (0.5,0.5,-2)
            -- (-1.5,-1.5,0)
            -- (-0.5,-0.5,2)
            -- cycle;
            
\draw[
        draw=green, draw opacity=0.4,
    ]
               (1.5,1.5,0)
            -- (0.5,0.5,-2)
            -- (-1.5,-1.5,0)
            -- (-0.5,-0.5,2);
            
\filldraw[
        draw=yellow!20, draw opacity=0.4,%
        fill=yellow!20,%
    ]          (0.5,-0.5,2)
            -- (1.5,-1.5,0)
            -- (0.5,-0.5,-2)
            -- (-0.5,0.5,-2)
            -- (-1.5,1.5,0)
            -- (-0.5,0.5,2)
            -- cycle;

\draw[
        draw=yellow, draw opacity=0.4, dotted
    ] (1.5,-1.5,0)
            -- (0.5,-0.5,-2)
            -- (-0.5,0.5,-2)
            -- (-1.5,1.5,0);

\draw[
        draw=yellow, draw opacity=0.4,
    ]  (1.5,-1.5,0)
            -- (0.5,-0.5,2)
            -- (-0.5,0.5,2)
            -- (-1.5,1.5,0);

\filldraw[
        draw=orange!20, draw opacity=0.4,%
        fill=orange!20,%
    ]          (-2,0.5,0.5)
            -- (0,1.5,1.5)
            -- (2,0.5,.5)
            -- (2,-0.5,-.5)
            -- (0,-1.5,-1.5)
            -- (-2,-0.5,-0.5)
            -- cycle;
 
\draw[
        draw=orange, draw opacity=0.4, dotted
    ] (2,-0.5,-.5)
    -- (0,-1.5,-1.5)
    -- (-2,-0.5,-0.5)
            -- (-2,0.5,0.5);

\draw[
        draw=orange, draw opacity=0.4,
    ]  (2,-0.5,-.5)
    -- (2,0.5,.5)
    -- (0,1.5,1.5)
            -- (-2,0.5,0.5);

\node (0) at (0,1,2) [vertex] {};
\node (1) at (0,2,1) [vertex] {};
\node (2) at (1,0,2) [vertex] {};
\node (3) at (1,2,0) [vertex] {};
\node (4) at (2,1,0) [vertex] {};
\node (5) at (2,0,1) [vertex] {};

\node (6) at (0,-1,2) [vertex] {};
\node (7) at (0,2,-1) [vertex] {};
\node (8) at (-1,0,2) [vertex] {};
\node (9) at (-1,2,0) [vertex] {};
\node (10) at (2,-1,0) [vertex] {};
\node (11) at (2,0,-1) [vertex] {};

\node (12) at (0,1,-2) [vertex] {};
\node (13) at (0,-2,1) [vertex] {};
\node (14) at (1,0,-2) [vertex] {};
\node (15) at (1,-2,0) [vertex] {};
\node (16) at (-2,1,0) [vertex] {};
\node (17) at (-2,0,1) [vertex] {};
\node (18) at (0,-1,-2) [vertex] {};
\node (19) at (0,-2,-1) [vertex] {};
\node (20) at (-1,0,-2) [vertex] {};
\node (21) at (-1,-2,0) [vertex] {};
\node (22) at (-2,-1,0) [vertex] {};
\node (23) at (-2,0,-1) [vertex] {};

\draw (1) -- (3); 
\draw (1) -- (9);
\draw (7) -- (3); 
\draw (7) -- (9);

\draw (0) -- (2); 
\draw (8) -- (6);
\draw (0) -- (8); 
\draw (6) -- (2);
 
\draw (4) -- (5); 
\draw (10) -- (11);
\draw (4) -- (11); 
\draw (5) -- (10);

\draw[dotted] (12) -- (14); 
\draw[dotted] (18) -- (20);
\draw[dotted] (12) -- (20); 
\draw[dotted] (18) -- (14);

\draw (13) -- (15); 
\draw[dotted] (21) -- (19);
\draw (13) -- (21); 
\draw[dotted] (15) -- (19);

\draw (16) -- (17); 
\draw[dotted] (22) -- (23);
\draw[dotted] (16) -- (23); 
\draw (17) -- (22);

\draw (0) -- (1);
\draw (6) -- (13); 
\draw[dotted] (7) -- (12);
\draw[dotted] (18) -- (19); 

\draw (2) -- (5);
\draw (8) -- (17); 
\draw[dotted] (14) -- (11);
\draw[dotted] (20) -- (23); 

\draw (3) -- (4);
\draw (9) -- (16); 
\draw (15) -- (10);
\draw (21) -- (22);

\end{tikzpicture}
\end{center}
\caption{Example of a mirror graph: cubic permutahedron with corresponding hyperplane arrangement}
\label{fig:permut}
\end{figure}

For an example consider the following classical problem. For a hyperplane $H_i$ with an orthogonal vector $v_i$ its \emph{reflection} is the map $\sigma_{H_i}(x)=x-2\frac{x\cdot v_i}{v_i\cdot v_i}v_i$. Arrangements of hyperplanes $\{H_1,\ldots,H_m\}$ in $\R^n$ such that all hyperplanes include vector 0, and for every $i\in \{1,\ldots,m\}$ the reflection of $H_i$ permutes the hyperplanes $\{H_1,\ldots,H_m\}$ are called \emph{reflection arrangements}. A simple example is a collection of $m$ vectors in a plane such that the angle between $v_i$ and a chosen axis is $\frac{i}{\pi}$, while for a more complicated example see Figure \ref{fig:permut}. 

A \emph{Coxeter group} is a group which can be presented by generators and relations as $\langle \alpha_1,\ldots,\alpha_m \mid (\alpha_i\alpha_j)^{k_{ij}}=1 \textrm{ for all } 1\leq i,j\leq m \rangle$, where $k_{ii}=1$ and $k_{ij}\geq 2$ for all $0\leq i<j\leq m$.
By a classical result \cite[Theorem 2.3.7]{MR1744046}, the reflection arrangements are in one to one correspondence with the finite Coxeter groups since the tope graphs of the reflection arrangements are the Cayley graphs of the finite Coxeter groups and vice versa. Moreover the finite Coxeter groups were classified by Coxeter \cite{coxeter1935complete}. They give rise to four infinite families and six exceptional cases of irreducible reflection arrangements. Here irreducible means that there is no non-trivial partition of hypercubes in two mutually orthogonal classes; equivalently, their tope graphs are not the Cartesian product of smaller tope graphs.

In this paper we characterize mirror graphs as precisely the Cayley graphs of the finite Coxeter groups and thus the tope graphs of reflection arrangements (see Theorem \ref{thm:main}). This not just fully classifies mirror graphs, but also gives a graph theoretical characterization of  reflection arrangements, a subproblem of a more general problem described above, and a characterization of the Cayley graphs of the finite Coxeter groups. Moreover, we provide a polynomial algorithm for the recognition of the three coinciding classes.
The rest of the paper is organized as follows. First we give some basic results, definitions and notations needed throughout the paper, while in the next section we prove the main result asserted in Theorem \ref{thm:main}.

First a few simple results about partial cubes. Relation $\Theta$  \cite{djokovic1973distance,winkler1984isometric} as defined above is an equivalence relation in partial cubes -- we will write $F_{uv}$ for the set of all edges that are in relation $\Theta$ with $uv$.
For an edge $uv$ in $G$ define $W_{uv}$ as the subset of vertices of $G$ that are closer to vertex $u$ than to $v$, that is $W_{uv}= \{ w: \   d(u,w) < d(v, w) \}  $. Notice that in a mirror graph $G$ sets $W_{uv},W_{vu}$ coincide with sets of vertices of $G_i^1$ and $G_i^2$, where $i$ is the index of $E_i$ for which $E_i=F_{uv}$. For the sake of consistency we will prefer the partial cubes notation over the mirror graphs notation. We will write $U_{uv}$ for the subset of vertices in $W_{uv}$ which have a neighbor in $W_{vu}$.  A path $P$ in a partial cube $G$ is a shortest path if and only if it has all of its edges in pairwise different $\Theta$-classes. For fixed $u,v$, all shortest $u,v$-paths pass the same $\Theta$-classes of $G$. If $C$ is a closed walk, then $C$ passes each $\Theta$-class an even number of times. For more information on the relation $\Theta$, we refer the reader to \cite[Chapter 11]{Hammack:2011a}.


We say that a subgraph $H$ of $G$ is \emph{convex} if all the shortest paths in $G$ between vertices in $H$ lie in $H$. 
In \cite{jaz} a \emph{convex traverse} was introduced as follows:
Let $v_1u_1\Theta v_2u_2$ in a partial cube $G$, with $v_2 \in U_{v_1u_1}$. Let $C_1,\ldots, C_n$, $n\geq 1$, be a sequence of convex cycles such that $v_1u_1$ lies only on $C_1$, $v_2u_2$ lies only on $C_n$, and each pair $C_i$ and $C_{i+1}$, for $i\in \{1,\ldots,n-1\}$, intersects in exactly one edge and this edge is in $F_{v_1u_1}$, all the other pairs do not intersect. If the shortest path from $v_1$ to $v_2$ on the union of $C_1,\ldots, C_n$ is a shortest $v_1, v_2$-path in $G$, then we call $C_1,\ldots, C_n$ a convex traverse from $v_1u_1$ to $v_2u_2$. In this case the shortest path from $u_1$ to $u_2$ on the union of $C_1,\ldots, C_n$ is also a shortest path and we call this two paths the \emph{sides} of a traverse. Most importantly, it was proved in \cite{jaz} that for arbitrary edges $v_1u_1, v_2u_2$ with $v_1u_1\Theta v_2u_2$ there exists a convex traverse connecting them.

 One of the concepts closely connected to oriented matroids is \emph{antipodality}. Call a graph $G$ \emph{even} if every vertex $v$ of $G$  has a unique \emph{antipodal} vertex $\bar{v}$ at the distance $\textrm{diam}(G)$ from $v$. Moreover, if for every adjacent vertices $u,v$, also antipodes $\bar{u},\bar{v}$ are adjacent, then call $G$ \emph{harmonic-even}. It follows from \cite{MR1210100} and \cite[Proposition 3.1]{klavzar2009even} that $G$ is harmonic-even partial cube if and only if every vertex $v$ has a vertex $\bar{v}$ at the distance $i(G)$ from $v$, where $i(G)$ denotes the \emph{isometric dimension} of $G$, i.e.~the number of $\Theta$-classes in $G$. 

We will consider the right actions of groups on the vertices of graphs and for an element $\alpha$ of a group $A$ acting on a graph $G$ we will denote by $v^{\alpha}$ the image of a vertex $v$ in $G$ by the action of $\alpha$. We will denote the Cayley graph of a group $A$, with generators $S=\{\alpha_1,\alpha_2,\ldots, \alpha_k\}$ such that $S^{-1}=S$ and $1 \notin S$, with $\textrm{Cay}(A,S)$, and interpret it as the graph with vertex set $A$ and two elements $\beta_1,\beta_2\in A$ adjacent if and only if $\beta_1=\alpha_i\beta_2$ for some $i\in \{1,\ldots,k\}$.

\section{Results}

We start by exposing a crucial property of mirror graph from which many properties will follow.

\begin{lem}\label{lem:miror_to_acycloid}
A mirror graph $G$ is harmonic-even.
\end{lem}

\begin{proof}
As mentioned in the preliminaries, it is enough to prove that every vertex $v\in G$ has a vertex $v'$ at the distance $i(G)$ from $v$, where $i(G)$ is the isometric dimension of $G$. 
Choose an arbitrary $v\in V(G)$ and let $u$ be a vertex that is at the maximal distance from $v$. By vertex transitivity, $d(v,u)=\diam(G)$. For the sake of contradiction assume that $d(v,u)<i(G)$, thus there exists a $\Theta$-class, say $F_{ab}$, such that $v,u \in W_{ab}$. Let $\alpha_{ab}$ be a mirror automorphism of $G$ that maps $W_{ab}$ to $W_{ba}$ with mapping every element of $U_{ab}$ to its neighbor.

Let $P$ be a shortest $v^{\alpha_{ab}},u$-path.  
Since $F_{ab}$ is a cut and $v^{\alpha_{ab}}\in W_{ba}$, $u\in W_{ab}$, there exists an edge $a'b'$ on $P$ that is in relation $\Theta$ with $ab$, where $a' \in U_{ab}$. The automorphism $\alpha_{ab}^{-1}$ maps the subpath of $P$ connecting $v^{\alpha_{ab}}$ and $b'$ to a path connecting $v$ and $a'$. The union of this path and the subpath of $P$ connecting $a'$ and $u$ is a $v,u$-path 
of length $d(u,v^{\alpha_{ab}})-1$. Thus  $d(v,u)\leq d(v^{\alpha_{ab}},u)-1\leq \diam(G)-1$. A contradiction.
\end{proof}

\begin{lem}\label{lem:inter}
Let $G$ be a harmonic-even partial cube, and $F_{ab}, F_{cd}$ its arbitrary $\Theta$-classes. Then there exists a convex cycle in $G$ that includes edges from $F_{ab}$ and $F_{cd}$.
\end{lem}

\begin{proof}
Let $vu$ be an arbitrary edge in $F_{ab}$ and $\bar{v},\bar{u}$ the antipodal vertices of $v,u$, respectively. Since $G$ is harmonic-even, $\bar{u}$ is adjacent to $\bar{v}$. Moreover, if $u\in W_{ba}$ and $v\in W_{ab}$, then $\bar{u} \in W_{ab}$ and $\bar{v} \in W_{ba}$ since the antipode of a vertex in a harmonic-even partial cube is at the distance $i(G)$ from it. Hence $\bar{u}\bar{v} \in F_{ab}$. Let $T$ be a convex traverse from $vu$ to $\bar{u}\bar{v}$. Since $d(v,\bar{v})=i(G)$,  all the $\Theta$-classes of $G$ cross $T$, thus there exists a convex cycle on $T$ that includes edges from $F_{cd}$. By definition of a traverse, the cycle also includes edges from $F_{ab}$.
\end{proof}


For the next lemma notice the following. Let $\alpha$ be an automorphism of a partial cube $G$. Relation $\Theta$ is defined based on distances in a graph, thus automorphisms of $G$ map $\Theta$-classes to $\Theta$-classes.
Assume that we know for each $\Theta$-class  of $G$ onto which $\Theta$-class it is mapped by $\alpha$, and additionally an image $v^{\alpha}$ of some vertex $v$ of $G$. Let $u$ be an arbitrary vertex of $G$ and $P$ a $v,u$-path. The knowledge of $\alpha$ determines where path $P$ is mapped by $\alpha$, since the beginning is mapped to $v^{\alpha}$ and each vertex is incident with at most one edge of each $\Theta$-class. In particular, the image of $u$ is fixed. Hence, the assumed knowledge completely determines $\alpha$.

\begin{lem}\label{lem:id}
Let $G$ be a harmonic-even partial cube and $F_{ab}$ a $\Theta$-class  in $G$. If $\alpha$ is an automorphism such that for each edge $vu \in F_{ab}$ it holds that $v^{\alpha}=v$ and $u^{\alpha}=u$, then $\alpha$ is the identity map.
\end{lem}

\begin{proof}
Let $\alpha$ be an automorphism of $G$ as described in the assertion. By the notice before the lemma and the fact that $v^{\alpha}=v$ for an arbitrary $v\in U_{ab}$, it suffice to show that every $\Theta$-class of $G$ is mapped to itself by $\alpha$. By definition, this holds for $F_{ab}$. On the other hand, if $F_{cd}$ is any $\Theta$-class in $G$, different from $F_{ab}$, let $C$ be a convex cycle that includes edges from $F_{ab}$ and  $F_{cd}$, provided by Lemma \ref{lem:inter}. Since $C$ is convex, it has two antipodal edges in $F_{ab}$. All four endpoints of the edges are mapped to itself by $\alpha$. Since $C$ is convex, this implies that all the vertices on $C$ must be mapped to itself. In particular, edges from $F_{cd}$ that lie on $C$ are mapped to itself, thus also $F_{cd}$ is mapped to itself.
\end{proof}

\begin{cor}\label{cor:mirror}
For a harmonic-even partial cube $G$ and a $\Theta$-class $F_{ab}$ in $G$, there exists at most one automorphism $\alpha_{ab}$ of $G$ such that for each $uv \in F_{ab}$ it holds $v^{\alpha_{ab}}=u$ and $u^{\alpha_{ab}}=v$. Moreover, $\alpha_{ab}^2=1$.
\end{cor}

\begin{proof}
Let $\alpha_{ab}$ be an automorphism from the assertion. For the  automorphism $\alpha_{ab}^2$ it holds that $v^{\alpha_{ab}^2}=v$, $u^{\alpha_{ab}^2}=u$ for every $vu\in F_{ab}$. Thus $\alpha_{ab}^2=1$, by Lemma \ref{lem:id}. If $\alpha'$ is an arbitrary automorphism of $G$ that maps every element of $U_{ab}$ to its neighbor in $U_{ba}$ and vice versa, then also for $\alpha_{ab}\alpha'$ holds that $v^{\alpha_{ab}\alpha'}=v$ and $u^{\alpha_{ab}\alpha'}=u$ for all $vu\in F_{ab}$. This implies $\alpha'=\alpha_{ab}^{-1}=\alpha_{ab}$.
\end{proof}


If $G$ is a mirror graph and $xy$ an edge in $G$, we can, due to Lemma \ref{lem:miror_to_acycloid} and Corollary \ref{cor:mirror}, denote with $\alpha_{xy}$ the unique \emph{mirror automorphism} of $G$: the automorphism that maps each vertex in $U_{xy}$ to its neighbor in $U_{yx}$ and vice versa. The uniqueness of it leads to  the following polynomial algorithm that for a graph $G$ with $n$ vertices and $m$ edges decides if $G$ is a mirror graph, and in the positive case outputs its mirror partition and mirror automorphisms:

\begin{algorithm}[ht]
\caption{Recognition of mirror graphs}
\begin{enumerate}
\item First check if $G$ is a partial cube by calculating the $\Theta$-classes and obtaining its embedding in a hypercube. This can be done in $O(n^2)$ by \cite{eppstein2011recognizing}. The $\Theta$-classes are candidates for the mirror partition of $G$. If $G$ is not a partial cube, it is not a mirror graph.
\item For each $\Theta$-class $F_{ab}$, its corresponding mirror automorphism $\alpha_{ab}$, if existent, must map all the convex cycles crossed by $F_{ab}$ to themselves. By Lemma \ref{lem:inter} this determines 
the image of each $\Theta$-class, and thus gives a candidate for the mirror automorphism. Convex cycles of $G$ can be found in $O(mn^2)$ by \cite{eppstein2011recognizing}, obtaining at most $O(nm)$ of them by \cite{azarija2015moore}. Iterating through convex cycles we can determine for each $\Theta$-class how its corresponding mirror automorphism permutes the other $\Theta$-classes. 
\item Considering $G$ embedded in a hypercube, each permutation of $\Theta$-classes can be seen as a permutation of coordinates of the hypercube that $G$ is embedded into, and thus it can easily be checked if the candidates for the mirror automorphisms in fact define automorphisms of $G$. If so, we output the $\Theta$-classes and the corresponding mirror automorphisms.
\end{enumerate}
\label{alg:recognition}
\end{algorithm}

\begin{lem}\label{lem:path}
Let $G$ be a mirror graph, $\alpha_{xy}$ an arbitrary mirror automorphism, and let $v'=v^{\alpha_{xy}}$ for a chosen $v\in V(G)$. Then there exists a path $P=vv_1v_2\ldots v_{n-1}v'$ from $v$ to $v'$, such that $\alpha_{xy}=\alpha_{vv_1}\alpha_{v_1v_2}\alpha_{v_2v_3}\ldots \alpha_{v_{n-1}v'}$.
\end{lem}

\begin{proof}
We will prove the lemma by induction on the distance from $v$ to $F_{xy}$, i.e.~the distance from $v$ to the closest edge that is in $F_{xy}$. It clearly holds if $v$ is incident with $F_{xy}$, i.e.~the distance is 0. Assume that the distance from $v$ to $F_{xy}$ is $d>0$, and that the lemma holds for all the vertices at the distance less than $d$. Take an arbitrary shortest path connecting $v$ and $F_{xy}$ and let $u$ be the neighbor of $v$ on it. Let $u'=u^{\alpha_{xy}}$. Notice that $\alpha_{xy}$ maps $F_{vu}$ to $F_{v'u'}$, by Corollary \ref{cor:mirror} it also maps $F_{v'u'}$ to $F_{vu}$. By the induction assumption, there exists a path $P'=uu_1u_2\ldots u_{n-1}u'$ from $u$ to $u'$, such that $\alpha_{xy}=\alpha_{uu_1}\alpha_{u_1u_2}\alpha_{u_2u_3}\ldots \alpha_{u_{n-1}u'}$. Therefore, to prove the lemma, it is enough to prove that $\alpha_{xy}=\alpha_{vu}\alpha_{xy}\alpha_{u'v'}$.

We will prove that $\alpha_{vu}\alpha_{xy}\alpha_{u'v'}\alpha_{xy}=1$. Consider  the image of $F_{vu}$ by $\alpha_{vu}\alpha_{xy}\alpha_{u'v'}\alpha_{xy}$. Automorhism $\alpha_{vu}$ maps $F_{vu}$ to itself, $\alpha_{xy}$ maps $F_{vu}$ to $F_{v'u'}$, by the above notice, $\alpha_{u'v'}$ maps $F_{v'u'}$ to itself, while $\alpha_{xy}$ maps $F_{v'u'}$ to $F_{uv}$. Thus $\alpha_{vu}\alpha_{xy}\alpha_{u'v'}\alpha_{xy}$ maps $F_{uv}$ to itself. Now take an arbitrary $w \in U_{uv}$, let $z$ be its neighbour in $U_{vu}$, and denote with $w'=w^{\alpha_{xy}}, z'=z^{\alpha_{xy}}$. Notice that $w'z'\in F_{v'u'}$. Then $\alpha_{vu}$ maps $w$ to $z$, $\alpha_{xy}$ maps $z$ to $z'$, $\alpha_{u'v'}$ maps $z'$ to $w'$, while $\alpha_{xy}$ maps $w'$ to $w$. By Lemma \ref{lem:id} and the fact that $w$ was an arbitrary element of $ U_{uv}$, it follows that $\alpha_{vu}\alpha_{xy}\alpha_{u'v'}\alpha_{xy}=1$. This finishes the proof.
\end{proof}

\begin{lem}\label{lem:cycle}
Let $C=(v_0v_1\ldots v_{2i-1})$ be a convex cycle in a mirror graph $G$. For every $j\in \{0,1,\ldots, 2i-1\}$ it holds $$\alpha_{v_{j+1}v_j}=\alpha_{v_1v_0}(\alpha_{v_0v_{2i-1}}\alpha_{v_1v_0})^{j}.$$
Moreover
$$\alpha_{v_0v_1}\alpha_{v_1v_2}\cdots \alpha_{v_{2i-1}v_0}=\alpha_{v_0v_{2i-1}}\alpha_{v_{2i-1}v_{2i-2}}\cdots \alpha_{v_{1}v_0}=1.$$
\end{lem}

\begin{proof}
Since $C$ is a convex cycle, it has its antipodal pairs of edges in relation $\Theta$. The latter implies that every mirror automorphism of a $\Theta$-class with edges on $C$ maps $C$ to $C$. We will first prove that $\alpha_{v_{j+1}v_j}=\alpha_{v_{j}v_{j-1}}\alpha_{v_{j-1}v_{j-2}}\alpha_{v_jv_{j-1}}$ for every $j\in \Z_{2i}$. Let $xy$ be an arbitrary edge from $F_{v_{j+1}v_j}$. Since $\alpha_{v_jv_{j-1}}$ maps $C$ to itself, it must map $v_{j+1}v_{j}$ to $v_{j-2}v_{j-1}$, thus it also maps $F_{v_{j+1}v_j}$ to $F_{v_{j-1}v_{j-2}}$. This implies that $\alpha_{v_jv_{j-1}}$ maps edge $xy$ to an edge in $F_{v_{j-1}v_{j-2}}$. Say $w=x^{\alpha_{v_jv_{j-1}}}$ and $z=y^{\alpha_{v_jv_{j-1}}}$. By definition of a mirror automorphism, $\alpha_{v_{j-1}v_{j-2}}$ maps $w$ to $z$ and vice versa. Thus $\alpha_{v_{j}v_{j-1}}\alpha_{v_{j-1}v_{j-2}}\alpha_{v_jv_{j-1}}$ maps the pair $(x,y)$ to the pair $(y,x)$ for every edge $xy\in F_{v_{j+1}v_j}$. By Corollary \ref{cor:mirror}, this automorphism must equal $\alpha_{v_{j+1}v_j}$.

Now we prove that $\alpha_{v_{j+1}v_j}=\alpha_{v_1v_0}(\alpha_{v_0v_{2i-1}}\alpha_{v_1v_0})^{j}.$  We will prove the assertion by induction on $j$. It clearly holds for $j=0$, while for $j=1$ we have $\alpha_{v_{2}v_1}=\alpha_{v_1v_0}\alpha_{v_0v_{2i-1}}\alpha_{v_1v_0}$ by the previous paragraph. We calculate:
\begin{align*}
\alpha_{v_{j+1}v_j}&=\alpha_{v_{j}v_{j-1}}\alpha_{v_{j-1}v_{j-2}}\alpha_{v_jv_{j-1}}\\
&=\alpha_{v_1v_0}(\alpha_{v_0v_{2i-1}}\alpha_{v_1v_0})^{j-1}\alpha_{v_1v_0}(\alpha_{v_0v_{2i-1}}\alpha_{v_1v_0})^{j-2}\alpha_{v_1v_0}(\alpha_{v_0v_{2i-1}}\alpha_{v_1v_0})^{j-1}\\
&=\alpha_{v_1v_0}\alpha_{v_0v_{2i-1}}(\alpha_{v_1v_0}\alpha_{v_0v_{2i-1}})^{j-2}\alpha_{v_1v_{0}}\alpha_{v_1v_0}(\alpha_{v_0v_{2i-1}}\alpha_{v_1v_0})^{j-2}\alpha_{v_1v_0}(\alpha_{v_0v_{2i-1}}\alpha_{v_1v_0})^{j-1}\\
&=\alpha_{v_1v_0}(\alpha_{v_0v_{2i-1}}\alpha_{v_1v_0})^{j}
\end{align*}
Notice that the latter implies that $\alpha_{v_0v_{2i-1}}=\alpha_{v_iv_{i-1}}=\alpha_{v_1v_0}(\alpha_{v_0v_{2i-1}}\alpha_{v_1v_0})^{i-1}$ (first equation holds since antipodal edges on $C$ are in relation $\Theta$), thus 
\begin{equation}\label{eq:1}
1=(\alpha_{v_0v_{2i-1}}\alpha_{v_1v_0})^{i}.
\end{equation}

Let $A_{C}$ be the subgroup of automorphisms of $G$ generated by the mirror automorphisms $\alpha_{v_1v_0},\ldots,\alpha_{v_{2i}v_{2i-i}}$. We have proved that $\alpha_{v_0v_{2i-1}}$ and $\alpha_{v_1v_0}$ generate $A_C$. Since (1) holds, $A_C$ must be a quotient of the group $\langle \alpha_{v_0v_1},\alpha_{v_0v_{2i-1}} \mid (\alpha_{v_0v_1}\alpha_{v_0v_{2i-1}})^{i}=1, \alpha_{v_0v_1}^2=1, \alpha_{v_0v_{2i-1}}^2=1 \rangle$, which is a Coxeter group of order $2i$. Since $A_{C}$ acts transitively on $C$, it must have at least $2i$ elements. Thus $A_{C}$ is isomorphic to $\langle \alpha_{v_0v_1},\alpha_{v_0v_{2i-1}} \mid (\alpha_{v_0v_1}\alpha_{v_0v_{2i-1}})^{i}=1, \alpha_{v_0v_1}^2=1, \alpha_{v_0v_{2i-1}}^2=1 \rangle$ which is precisely the group of all edge symmetries of $C$. In particular it holds,
$$\alpha_{v_0v_1}\alpha_{v_1v_2}\cdots \alpha_{v_{2i-1}v_0}=\alpha_{v_0v_{2i-1}}\alpha_{v_{2i-1}v_{2i-2}}\cdots \alpha_{v_{1}v_0}=1.$$
%
%
\end{proof}

\begin{lem}\label{lem:matroid}
Let $C=(v_0v_1\ldots v_{2i-1})$ be a convex cycle in a mirror graph. Then any edge $ab \in F_{v_0v_1}$ is in the intersection of $W_{v_2v_1},W_{v_3v_2},\ldots, W_{v_iv_{i-1}}$ or in the intersection of the complements of these sets.
\end{lem}

\begin{proof}
Assume that $ab \in F_{v_0v_1}$ is in $W_{v_2v_1}$. Consider a convex traverse from $v_0v_1$ to $ab$. Since $v_0,v_1 \in W_{v_1v_2}$, $a,b \in W_{v_2v_1}$ and $F_{v_1v_2}$ is a cut, there must be a convex cycle $D$ on $T$, that includes an edge from $F_{v_1v_2}$. Pick an arbitrary $v_{j-1}v_{j}$ for $3\leq j\leq i$. Then for $k=\lfloor j/2 \rfloor$, $\alpha_{v_kv_{k+1}}$ maps $v_{j-1}v_{j}$ to $v_0v_1$ or $v_1v_2$ (and vice versa)  since it maps $C$ to $C$. Hence it maps $F_{v_0v_1}$ or $F_{v_1v_2}$ to $F_{v_{j-1}v_{j}}$. On the other hand, by Lemma \ref{lem:cycle}, $\alpha_{v_kv_{k+1}}$ can be expressed as a combination of automorphisms $\alpha_{v_0v_{1}}$ and $\alpha_{v_1v_{2}}$. Since each such automorphism maps $D$ to $D$, also $\alpha_{v_kv_{k+1}}$ maps $D$ to $D$. Thus there must be an edge on $D$ in $F_{v_{j-1}v_{j}}$. This implies $a,b\in W_{v_jv_{j-1}}$.

On the other hand, $ab \in F_{v_0v_1}$ can lie in $W_{v_1v_2}=W_{v_{i+2}v_{i+1}}$. Now the result follows if we consider a traverse from $v_iv_{i+1}$ to $ab$ and use the same arguments.
\end{proof}

One could easily deduce from Lemma \ref{lem:matroid}, using results from \cite{handa1990characterization}, that every mirror graph is in fact a tope graph of an oriented matroid. We will not make this argumentation since this result will follow from the main theorem. 

Denote with $\mathcal{C}(G)$ the 2-dimensional cell complex whose 2-cells are obtained by replacing each convex cycle $C$ of length $2j$ of $G$ by a regular Euclidean polygon $[C]$ with $2j$ sides. In \cite{chepoi2016partial} it was proved that for a partial cube $G$, the complex $C(G)$ is simply connected.

\begin{thm}\label{thm:main}
For a graph $G$ the following statements are equivalent:
\begin{enumerate}[(i)]
\item $G$ is a mirror graph.
\item $G$ is the Cayley graph of a finite Coxeter group.
\item $G$ is the tope graph of a reflection arrangement.
\end{enumerate}
\end{thm}

\begin{proof}
The crucial part is to prove that $\textrm{(i)} \implies \textrm{(ii)}$. Assume that $G$ is a mirror graph. We will first prove that $G$ is a Cayley graph. Let $A$ be the subgroup of $\textrm{Aut}(G)$ generated by all the mirror automorphisms. Group $A$ acts transitively on the vertices of $G$. Recall that by a theorem of Sabidussi \cite{sabidussi1958class}, $G$ is a Cayley graph of a group $A$ if $A$ acts transitively on the vertices of $G$ and the stabilizers of the vertices are trivial. Therefore, to prove the assertion it suffice to prove that for an arbitrary vertex $v\in V(G)$ its stabilizer is trivial. Assume that an automorphism $\alpha_{a_1b_1}\alpha_{a_2b_2}\ldots \alpha_{a_nb_n}$ maps $v$ to itself. Let $v_1=v^{\alpha_{a_1b_1}}$, $v_2=v_1^{\alpha_{a_2b_2}}$, $\ldots$, $v=v_{n-1}^{\alpha_{a_nb_n}}$. By Lemma \ref{lem:path}, $\alpha_{a_1b_1}\alpha_{a_2b_2}\ldots \alpha_{a_nb_n}$ equals to $\alpha_{vu_1}\alpha_{u_1u_2}\ldots \alpha_{u_{m-1}v}$ where $(vu_1u_2u_3\ldots u_{m-1})$ is a closed walk from $v$ to $v$ passing $v_1,v_2,\ldots, v_{m-1}$. Graph $G$ is a partial cube, hence its 2-dimensional cell complex $\mathcal{C}(G)$ made out of the convex cycles is simply connected. Since for each convex cycle $(a_0a_1\ldots a_{2i-1})$ in $G$, by Lemma \ref{lem:cycle}, holds that $a_{a_0a_1}a_{a_1a_2}\ldots a_{a_{2i-1}a_0}=1$, the latter implies that we can transform $a_{vu_1}a_{u_1u_2}\ldots a_{u_{m-1}v}$ to the identity using equalities on convex cycles. Thus the stabilizer of $v$ is trivial, and $G$ is a Cayley graph.

Now pick an arbitrary vertex $v\in V(G)$ and let $vv_1,vv_2,\ldots,vv_k$ be the edges incident with it. Then $G\cong \textrm{Cay}(A,\{\alpha_{vv_1},\ldots,\alpha_{vv_k}\})$. We want to identify group $A$ to understand the structure of $G$. First identify vertices of $G$ with elements of $A$ in the standard way: identify the chosen vertex $v$  with the identity 1 of $A$ and every vertex $u$ of $G$ with the unique automorphism $\alpha\in A$ such that $u=v^{\alpha}$.  By definition of the Cayley graph, every relation of generators, say $\alpha_{vv_{i_p}}\ldots\alpha_{vv_{i_2}} \alpha_{vv_{i_1}}=1$, gives us a closed walk on vertices $1,\alpha_{vv_{i_1}},\alpha_{vv_{i_2}}\alpha_{vv_{i_1}},\alpha_{vv_{i_3}}\alpha_{vv_{i_2}}\alpha_{vv_{i_1}}, \ldots,\alpha_{vv_{i_{p-1}}}\ldots \alpha_{vv_{i_2}}\alpha_{vv_{i_1}}$. 
For the latter we will say that a relation in $A$ generates the closed walk in $G$.
First we prove the following claim:

\begin{claim}\label{claim:cycle}
Every pair of incident edges $vv_i, vv_j$ lies on the unique convex cycle $C$ generated by the relation $(a_{vv_i}a_{vv_j})^{k_{ij}}=1$, where $k_{ij} \geq 2$ equals half of the length of $C$.
\end{claim}

\begin{proof}
Let $C$ be the closed walk generated by the relation $(a_{vv_i}a_{vv_j})^{k_{ij}}=1$, where $k_{ij}$ is as small as possible ($k_{ij}$ exists since $A$ is finite).
First we want to identify $\Theta$-classes of edges that lie on $C$.
Let $D$ be a convex cycle in $G$ with nontrivial intersection with $F_{vv_i}$ and $F_{vv_j}$, provided by Lemma \ref{lem:inter}.  Denote with $xy$ an edge on $D$ that is $F_{vv_i}$, and without loss of generality assume that $D$ and $xy$ are such that the distance between $vv_i$ and $xy$ is as small as possible. 
There are two antipodal edges on $D$ that are in $F_{vv_j}$, let $wz$ be the one that is in $W_{vv_i}$.
 
Now we prove that $xy$ and  $wz$ are incident.
Let $T_i$ be a convex traverse connecting $vv_i$ and $xy$, and $T_j$ a convex traverse connecting $vv_j$ and $wz$. There is no edge in $F_{vv_j}$ on $T_i$ since otherwise there would exist a convex cycle with a nontrivial intersection with $F_{vv_i}$ and $F_{vv_j}$ but closer to $vv_i$ than $D$. Also, there is no edge in $F_{vv_i}$ on $T_j$ since $vv_j,wz \in W_{vv_i}$. Assume that there is an edge $ab$ on the shortest path connecting $xy$ and $wz$ on $D$, say $xy\in W_{ab}$. Then the sides of $T_i$ and $T_j$, together with the shortest path connecting $xy$ and $wz$ form a closed walk. Since a closed walk must pass each $\Theta$-class even number of times, this implies that one of $T_i,T_j$ has an edge in $F_{ab}$. Without loss of generality assume that $T_i$ has an edge in $F_{ab}$. In this case the edge $vv_i$ is in $W_{ba}$, but not in $W_{v_jv}$. A contradiction with Lemma \ref{lem:matroid}. Thus $xy$ and  $wz$ are incident.

Denote with $(u_0u_1\ldots u_{2k-1})$ the vertices of $D$, where $wz=u_0u_1$, $xy=u_0u_{2k-1}$ and $2k$ is the length of $D$.
First notice that $\alpha_{vv_j}=\alpha_{u_0u_1}$ and $\alpha_{vv_i}=\alpha_{u_0u_{2k-1}}$ since $vv_j\Theta u_0u_1$ and $vv_i\Theta u_0u_{2k-1}$. By Lemma \ref{lem:cycle}, the smallest $k_{ij}$ such that $(\alpha_{u_0u_{2k-1}}\alpha_{u_0u_{1}})^{k_{ij}}=1$ 
is $k$, i.e.~$k_{ij}$ is half the length of $D$.
Thus also $C$ has length $2k$. 
%

Edges on $C$ are
 connecting vertices of the form $(\alpha_{vv_i}\alpha_{vv_j})^{l}$ and $\alpha_{vv_j}(\alpha_{vv_i}\alpha_{vv_j})^{l}$, 
 or $\alpha_{vv_j}(\alpha_{vv_i}\alpha_{vv_j})^{l}$ and $(\alpha_{vv_i}\alpha_{vv_j})^{l+1}$, for some $0\leq l< k$. Automorphism $\alpha_{vv_j}(\alpha_{vv_i}\alpha_{vv_j})^{l}$ (by the right action) maps vertex $(\alpha_{vv_i}
 \alpha_{vv_j})^{l}$ to $\alpha_{vv_j}$ and vertex $\alpha_{vv_j}(\alpha_{vv_i}\alpha_{vv_j})^{l}$ to $	1$. 
 But $\alpha_{vv_j}(\alpha_{vv_i}\alpha_{vv_j})^{l}=\alpha_{u_0u_1}(\alpha_{u_0u_{2k-1}} \alpha_{u_0u_1})^l$, and, by Lemma \ref{lem:cycle},  $\alpha_{u_0u_1}(\alpha_{u_0u_{2k-1}} \alpha_{u_0u_1})^l=\alpha_{u_{l+1}u_l}$. Automorphism $\alpha_{u_{l+1}u_l}$ maps $F_{u_{2l+1}u_{2l}}$ to $F_{u_0u_1}$. Since the edge between $(\alpha_{vv_i}\alpha_{vv_j})^{l}$ and $\alpha_{vv_j}
 (\alpha_{vv_i}\alpha_{vv_j})^{l}$ is mapped to edge between 1 and $vv_j$ which is in $F_{u_0u_1}$, it must be in $F_{u_{2l+1}u_{2l}}$.

 
Similar analysis can be made for the edges between $\alpha_{vv_j}(\alpha_{vv_i}\alpha_{vv_j})^{l}$ and $(\alpha_{vv_i}\alpha_{vv_j})^{l+1}$, for $0\leq l< k$.  Thus all the edges on $C$ are in the same $\Theta$-classes as edges on $D$. Now we can prove that $C$ is convex. Let $ab$ be an edge on $C$ different from $vv_i$ but in $F_{vv_i}$. Consider a convex traverse $T$ from $vv_i$ to $ab$. Since there is a path on $C$ connecting both edges that includes only edges that are in $\Theta$ relation with edges on $D$, we deduce that also all the edges on $T$ are in $\Theta$ relation with edges on $D$. By Lemma \ref{lem:matroid}, it follows that $T$ must be a single convex cycle $E$ of length the same as $D$, that is $2k$. 

For the uniqueness, we shall prove that no two convex cycles can share more than an edge or a vertex. Assume that two different convex cycles $D_1,D_2$ share two vertices. Since they are convex, they share a shortest path connecting this two vertices. Now, assume that there are at least two edges on this path. Let $x_0x_1x_2$ be a subpath that is shared by both, such that $x_0$ is also incident with non-identical edges $x_0y_1$ and $x_0y_2$, such that they lie on $D_1\backslash D_2$ and $D_2\backslash D_1$, respectively. The mirror automorphism $\alpha_{x_0x_1}$ maps $D_1$ onto $D_1$ and $D_2$ onto $D_2$. But this is impossible since $x_1x_2$ cannot simultaneously get mapped to $x_0y_1$ and $x_0y_2$.
\end{proof}

 We claim that $A=\langle \alpha_{vv_1},\alpha_{vv_2}, \ldots , \alpha_{vv_k} \mid (\alpha_{vv_i}\alpha_{vv_j})^{k_{ij}}=1 \rangle$, where $k_{ii}=1$ and $k_{ij}\geq 2$ is given by Claim \ref{claim:cycle}. Since $G$ is a connected Cayley graph of $A$, $A$ is generated by $\alpha_{vv_1},\alpha_{vv_2}, \ldots , \alpha_{vv_k}$ and the relations hold by Claim \ref{claim:cycle} and Lemma \ref{lem:cycle}. Assume that for the above generators some other relation holds, say $\alpha_{vx_1}\alpha_{vx_2} \ldots  \alpha_{vx_j}=1$ for $x_i\in \{v_1,\ldots,v_k\}$ for all $i\in \{1,\ldots,j\}$. We want to prove that this relation can be derived from the above relations. Again identify vertices of $G$ with elements of $A$ in the standard way.
 
\begin{claim}\label{claim:relations}
Let $D$ be a convex cycle in $G$. Then $D$ lies on vertices $\alpha_{vv_i}a$, $\alpha_{vv_j}\alpha_{vv_i}a$, $\alpha_{vv_i}\alpha_{vv_j}\alpha_{vv_i}a$,\ldots, $(\alpha_{vv_j}\alpha_{vv_i})^{k_{ij}}a$, for some generators $\alpha_{vv_i},\alpha_{vv_j}$ and some $a\in A$.
\end{claim} 
\begin{proof}
Pick a vertex $u$ on $D$ and let $a\in A$ be the map that maps $v$ to $u$, i.e.~$u$ is identified with $a$. Let $\alpha_{vv_i}a,\alpha_{vv_j}a$ be the neighbors of $a$ on $D$. Then automorphism $a^{-1}\in A$ maps $D$ to a convex cycle $C$ incident with $1,\alpha_{vv_i},\alpha_{vv_j}$. By Claim \ref{claim:cycle} the only such cycle $C$ is the cycle on vertices $\alpha_{vv_i}$, $\alpha_{vv_j}\alpha_{vv_i}$, $\alpha_{vv_i}\alpha_{vv_j}\alpha_{vv_i}$,\ldots, $(\alpha_{vv_j}\alpha_{vv_i})^{k_{ij}}$. Automorphism $a$ maps $C$ to $D$ which proves the claim.
\end{proof}

Assume that the relation $\alpha_{vx_1}\alpha_{vx_2} \ldots  \alpha_{vx_j}=1$ holds in $A$. As above it generates a closed walk in $G$. Since $G$ is a partial cube, it 2-dimensional cell complex $\mathcal{C}(G)$ is simply connected, thus every closed walk is generated by convex cycles in $G$. In the language of the generators this implies that all the relations in $G$ can be derived from relations on the convex cycles of $G$. By Claim \ref{claim:relations}, the relations on the convex cycles of $G$ are derived from asserted relations.

We have proved that $A=\langle \alpha_{vv_1},\alpha_{vv_2}, \ldots , \alpha_{vv_k} \mid (\alpha_{vv_i}\alpha_{vv_j})^{k_{ij}}=1 \rangle$. Thus $A$ is a finite Coxeter group and $\textrm{(i)} \implies \textrm{(ii)}$ follows. As described in the preliminaries, the Cayley graphs of the finite Coxeter groups are in one to one correspondence with the tope graphs of the reflection arrangements thus $\textrm{(ii)} \iff \textrm{(iii)}$.

Finally, assume that $G$ is the tope graph of a reflection arrangement.  Then we can partition the edges of $G$ into sets corresponding to the hyperplanes in the arrangement. Moreover, the reflection of each hyperplane  maps chambers to chambers and hyperplanes to hyperplanes, and thus induces  a mirror automorphism of $G$. We deduce that $G$ is a mirror graph and $\textrm{(iii)}\implies \textrm{(i)}$ follows.
\end{proof}

Theorem \ref{thm:main} implies that via Algorithm \ref{alg:recognition} also the Cayley graphs of the finite Coxeter groups and the tope graphs the of the reflection arrangements can be recognized in polynomial time.

\section*{Acknowledgment}
Author wishes to express his gratitude to Sandi Klavžar for the useful comments on the text.

\section*{References}
\bibliographystyle{plain}
\bibliography{biblio}

\end{document}